\documentclass[a4paper, 11pt]{amsart}
\usepackage{enumitem}
\usepackage{amsmath}
\usepackage{amsthm}
\usepackage{amssymb}
\usepackage{comment}
\usepackage{graphicx}
\usepackage{todonotes}
\usepackage{hyperref}
\usepackage[capitalize]{cleveref}
\usepackage[normalem]{ulem}

\newcounter{alph}

\newtheorem{theo}[alph]{Theorem}

\numberwithin{equation}{section}
\newtheorem{cor}[equation]{Corollary}
\newtheorem{lem}[equation]{Lemma}

\newtheorem{prop}[equation]{Proposition}
\newtheorem{thm}[equation]{Theorem}
\theoremstyle{definition}

\newtheorem{rem}[equation]{Remark}

\DeclareMathOperator{\im}{im}

\DeclareMathOperator{\cus}{cus}

\DeclareMathOperator{\Lip}{Lip}
\DeclareMathOperator{\rk}{rank}
\DeclareMathOperator{\Ray}{Ray}
\DeclareMathOperator{\res}{res}

\DeclareMathOperator{\sys}{sys}

\renewcommand{\Re}{\operatorname{Re}}
\renewcommand{\Im}{\operatorname{Im}}

\def\C{\mathbb C}

\def\N{\mathbb N}
\def\R{\mathbb R}
\def\H{\mathbb H}
\def\Z{\mathbb Z}

\def\ve{\varepsilon}
\def\vf{\varphi}
\def\la{\langle}
\def\ra{\rangle}

\definecolor{blue(ncs)}{rgb}{0.0, 0.53, 0.74}
\definecolor{green}{rgb}{0.0, 0.5, 0.35}

\begin{document}

\title[Small eigenvalues of pseudo-Laplacians]{Small eigenvalues of pseudo-Laplacians}
\author[Ballmann]{Werner Ballmann}
\address
{WB: Max Planck Institute for Mathematics,
Vivatsgasse 7, 53111 Bonn}
\email{hwbllmnn\@@mpim-bonn.mpg.de}
\author[Mondal]{Sugata Mondal}
\address{SM: Department of Mathematics and Statistics, University of Reading, Pepper Lane, Whiteknights, RG6 6AX UK}
\email{s.mondal@reading.ac.uk}
\author[Polymerakis]{Panagiotis Polymerakis}
\address{PP: Department of Mathematics, University of Thessaly, 3rd km Old National Road Lamia–Athens, 35100, Lamia, Greece}
\email{ppolymerakis\@@uth.gr}

\date{December 22, 2025}



\thanks{\emph{Acknowledgments.}
We are grateful to the Max Planck Institute for Mathematics
and the Hausdorff Center for Mathematics in Bonn for their support and hospitality. S. M. was partially supported by EPSRC grant APP16691.}
\maketitle

\begin{abstract}
We extend the Otal-Rosas bound on the number of small eigenvalues of the Laplacian on a hyperbolic surface to the small eigenvalues of pseudo-Laplacians. In the process, we extend the work of Colin de Verdi\`ere on the spectral theory of pseudo-Laplacians to hyperbolic surfaces with more than one cusp.
\end{abstract}

\section{Introduction}

The spectral theory of Laplacians of hyperbolic surfaces (normalized to have curvature $-1$) is a classical topic in Riemannian geometry
with connections ranging from number theory to mathematical physics.
The universal covering of a hyperbolic surface is the hyperbolic plane $\H$,
and the bottom of the spectrum of the Laplacian of $\H$ is $1/4$.

A hyperbolic surface $S$ is obtained as the quotient $\Gamma\backslash\H$ of $\H$ by a discrete and torsionless group $\Gamma$ of isometries of $\H$.
We are interested in the case where the area $|S|<\infty$.
Then the spectrum of (the Laplacian $\Delta$ of) $S$ is discrete below $1/4$,
and eigenvalues below or equal to $1/4$ are called \emph{small}.

For the closed surface $S_g$ of genus $g$, the following optimal estimate was conjectured by Buser and Schmutz \cite{Bu,Sch}.
The conjecture was proved by Otal and Rosas in the more general case of surfaces $S_{g,n}$, that is, surfaces $S_g$ with $n\ge0$ punctures \cite[Th\'eor\`eme 2]{OR}:

\begin{theo}\label{otro}
A hyperbolic metric on $S_{g,n}$ of finite area has at most $|\chi(S)|$ small eigenvalues.
\end{theo}

Any non-compact hyperbolic surface $S$ of finite area is diffeomorphic to a closed surface $\bar S$ with finitely many points  $p_1,\dots,p_n$, called \emph{punctures}, removed.
Each puncture $p_i$ has a neighborhood $C_i$ in $S$, called a \emph{cusp}, which is isometric to a quotient $\mathbb{Z}\backslash\H$, truncated from below,
where we view $\H$ as the upper half-plane, $\H=\{z=x+iy\mid y>0\}$, and the infinite cyclic group $\mathbb{Z}$ (of integers) as acting on $\H$ by shifting horizontally.
Then the \emph{cuspidal} part of the spectrum of $S$, that is, the part which comes from eigenfunctions on $S$ which vanish asymptotically at the punctures, is of interest.
Such eigenfunctions and their eigenvalues are also called \emph{cuspidal} or \emph{parabolic}.
Cuspidal eigenfunctions can be extended by zero to $\bar S$.
By Huxley \cite[Theorem, p.\,352]{Hu} and Otal \cite[Proposition 2]{Ot} (using a different method),  $S$ does not have any small cuspidal eigenvalues if $\bar S$ is a sphere or a torus.
Furthermore, Otal shows that the (cuspidal) multiplicity of a small cuspidal eigenvalue is at most $-\chi(\bar S)-1$ \cite[Proposition 3]{Ot}.
Finally, Otal and Rosas conjecture that a non-compact hyperbolic surface $S$ of finite area
with $n$ punctures has at most
\begin{align}\label{otroc}
    -\chi(\bar S)-1=-\chi(S)-n-1
\end{align}
small cuspidal eigenvalues \cite[p.\,113]{OR}.
Partial results towards this conjecture were obtained in \cite{M2}.
In particular, by \cite[Theorem 1.6]{M2}, the conjectured estimate holds true for $S$ if certain closed geodesics on $S$ are sufficiently short.
The conjecture as such is open and is one of the spectral problems in the background of the present work.

We will study an approximation of the spectrum of $S$,
introduced by Lax and Phillips \cite{LP} and investigated further by Colin de Verdi\`ere \cite{CV}.
To describe it, we fix some notation.
We let $S$ be a non-compact hyperbolic surface of finite area,
diffeomorphic to a compact surface $\bar S$ with $n$ punctures $p_1,\ldots p_n$.
As cusps about the $p_i$, we choose pairwise disjoint closed neighborhoods $C_{i,b_i}$ in $S$, isometric to quotients of horoballs $\{z=x+iy\mid y\ge b_i\}$ in the upper half-plane $\H$ by shifts of the $x$-variable by integers (a normalization).
We let $b=(b_1,\dots,b_n)$.

For $a_i\ge b_i$, the \emph{horocycles} $H_{i,a_i}\cong\{(x+ia_i)\}\subseteq C_{i,b_i}$ and cusps $C_{i,a_i}\cong\{(x+iy)\mid y\ge a_i\}\subseteq C_{i,b_i}$
are quotients of horocycles and horoballs in $\H$.
For any tuple $a=(a_1,\dots,a_n)$ of real numbers with $a>b$, that is, $a_i>b_i$ for all $1\le i\le n$, we let
\begin{align}\label{haca}
    H_a = \cup_i H_{i,a_i} \quad\text{and}\quad C_a = \cup_i C_{i,a_i}.
\end{align}
For later purposes, we also introduce $S_a$, the compact domain in $S$ with boundary $H_a$ such that the interior of $S_a$ is the complement of $C_a$,
\begin{align}\label{saca}
    \Gamma\backslash\mathbb{H} = S = S_a\cup C_a \quad\text{and}\quad S_a \cap C_a = H_a.
\end{align}
Consider now the closed subspace $\mathcal{H}_a$ of the Sobolev space $\mathcal{H}_\infty=H^1(S)$, given by functions $f$ whose first Fourier coefficient $[f]_{i}=[f]_i(y)$,
with respect to the periodic variable $x$ on $\R/\Z$, vanishes on $C_{i,a_i}$ for all $1\le i\le n$.
The closure $\mathcal{L}_a$ of $\mathcal{H}_a$ in $\mathcal{L}_\infty=L^2(H)$  is the orthogonal complement of the space of functions in $\mathcal{L}_\infty$ which vanish outside $C_a$ and depend only on $y$ inside $C_a$.
The Friedrichs extension $\Delta_a$ of the Dirichlet form $\|\nabla f\|^2$ on $\mathcal{H}_a$ is a self-adjoint operator in $\mathcal{L}_a$ and is called a \emph{pseudo-Laplacian}; compare with \cite{CV} and \cref{prelim2} below.
Furthermore, the Laplacian $\Delta$ coincides with the Friedrichs extension $\Delta_\infty$ of the Dirichlet form $\|\nabla f\|^2$ on $\mathcal{H}_\infty$.
The resolvents of the $\Delta_a$ are compact and, therefore, their spectra are discrete \cite[p.\,206 ff]{LP}.
The eigenvalues of $\Delta_a$ approximate the eigenvalues of $\Delta=\Delta_\infty$ in the interval $[0,1/4]$ monotonically in $a$ as $a$ tends to $\infty$; see \cref{thmy} below. 
For any $\lambda\ge0$, denote by $N_a(\lambda)$ the number of eigenvalues of $\Delta_a$ in $[0,\lambda]$, including multiplicity.
For $\lambda\le1/4$, denote by $N(\lambda)$ the number of eigenvalues of $\Delta$  in $[0,\lambda]$, also including multiplicity.

\begin{theo}\label{mainq}
For any given $\lambda\ge0$, $N_a(\lambda)$ is non-decreasing in $a$. 
Moreover,
\begin{enumerate}
    \item\label{nal} if $\lambda\le1/4$, then $N_a(\lambda)\le N(\lambda)\le|\chi(S)|$;
    \item\label{nai} if $\lambda>1/4$, then $\lim_{a\to\infty} N_a(\lambda)=\infty$.
\end{enumerate}
\end{theo}

The estimate against $|\chi(S)|$ follows from \cite[Th\'eor\`eme 2]{OR} in the orientable and from \cite[Theorem 1.5]{BMM} in the general case, the other assertions follow from monotonicity and \cref{specs} below.

One of the main ingredients in the work of Otal \cite{Ot} and Otal-Rosas \cite{OR} is the topology of nodal domains and estimates on the bottom of their Dirichlet spectrum when their `topology is small'.
This led to the definition of \emph{analytic systole} $\Lambda(S)$ in \cite{BMM}.
The corresponding concept in our context is caught by the definition to follow.

Let $\Lambda_{a}(S)$ be the infimum of the bottom of the Dirichlet spectrum over all embedded topological discs, M\"obius bands (if $S$ is non-orientable), and annuli in $S$ with smooth boundaries, where the annuli are of the following three types:
Either their boundary curves are inessential (i.e., homotopic to a point), or they are freely homotopic to a closed geodesic in $S$, or they are freely homotopic to a horocycle $H_{a_i}$ for some $1\le i\le n$.
In the last case we require in addition that, for each horocycle $H_y$ above $H_a$, if the interior of such an annulus intersects $H_y$, then its boundary meets $H_y$ in at least two points.
The latter property reflects the condition on the Fourier coefficients of functions in the domain of $\Delta_a$.
Our first main result is as follows.

\begin{theo}\label{mainest}
For any $a>b$, we have $\Lambda_{a}(S)>1/4$.
\end{theo}

In contrast, the analytic systole $\Lambda(S)=1/4$ for noncompact hyperbolic surfaces of finite area (or even finite type).
Extending the ideas in the proof of Theorem 1.5 in \cite{BMM}, we obtain our second main result.

\begin{theo}\label{main}
For any $a>b$, the associated \emph{pseudo-Laplacian} $\Delta_a$ has at most $|\chi(S)|$ eigenvalues in $[0,\Lambda_a(S)]$.
\end{theo}

Since $\Lambda_a(S)>1/4$, \cref{main} sharpens \cref{mainq}.\ref{nal}.

\subsection{Structure of the article}\label{substr}
In \cref{prelim2}, we discuss Friedrichs extensions in general and in the case of pseudo-Laplacians $\Delta_a$.
In \cref{speclap}, we collect material from the spectral theory of Laplacians of finite area hyperbolic surfaces.
For the convenience of the reader, we recall some properties of Eisenstein series from Iwaniec's book \cite{Iw}.
\cref{subcdv} is devoted to the spectral theory of pseudo-Laplacians.
The case of hyperbolic surfaces with more than one cusp presents some new features and difficulties in comparison with the case of one cusp, considered in \cite{CV}.
Theorems \ref{thmxx}, \ref{thmyy}, and \ref{proyy}, in which we clarify the structure of eigenfunctions, bear witness of these new features.
Albeit not all our results in \cref{subcdv} are needed in the proofs of Theorems \ref{mainest} and \ref{main}, they should be useful in further investigations and applications of the spectral theory of pseudo-Laplacians.
In \cref{secana}, we discuss the proof of \cref{mainest},  in \cref {secnod} the proof of \cref{main}.

\section{Preliminaries}\label{prelim2}

We recall the following version of the Friedrichs extension; see \cite[p.\,600 f]{T1}.
Let $J\colon H^1\rightarrow H^0$ be a continuous inclusion of Hilbert spaces with $H^1$ dense in $H^0$.
Obtain a self-adjoint operator $A$ in $H^0$ with domain
\begin{align}\label{doma}
    \mathcal{D}(A) = \{ u\in H^1 \mid \text{$\la u,v\ra_1 \le C_u|v|_0$ for all $v\in H^1$} \}
\end{align}
by the defining equation
\begin{align}\label{equa}
    \la Au,v\ra_0 = \la u,v\ra_1.
\end{align}
We have $A=(JJ^*)^{-1}$ with $\mathcal{D}(A)=\im(JJ^*)$.
Furthermore, the closed sesquilinear form $\alpha$ associated to $A$ has domain $H^1$ and is given by
\begin{align}\label{ses1}
    \alpha(u,v) = \la u,v\ra_1 - \la u,v\ra_0 = \la (A-1)u,v\ra_0,
\end{align}
where we need $u\in\mathcal{D}(A)$ for the latter equality;
compare with \cite[Theorem VI.2.1]{Ka95}.

We recall now the general setup in \cite[Section 1]{CV}.
Let $H^1$ be a closed subspace of $H^1(S)$, $H^0$ be its closure in $ L^2(S)$,
and $A=1+\Delta_{H^1}$ be the self-adjoint operator in $H^0$ as above.
Since $H^1$ is a closed subspace of $H^1(S)$,
the closed sesquilinear form $\alpha$ associated to $A$ has domain $H^1$ and is given by
\begin{align}\label{ses2}
    \alpha(u,v) = \la\nabla u,\nabla v\ra_0 = \la\Delta_{H^1}u,v\ra_0,
\end{align}
where we need $u\in D(\Delta_{H^1})$ for the latter equality.
We have
\begin{align}\label{calh}
\begin{split}
    \la u,v\ra_0 +\la\nabla u,\nabla v\ra_0
    &=\la u,v\ra_1
    = \la u,v^\top + v^\bot \ra_1
    = \la u,v^\top\ra_1 \\
    &= \la Au,v^\top\ra_0
    = \la u,v^\top\ra_0 + \la\Delta_{ H^1}u,v^\top\ra_0,
 \end{split}
\end{align}
for all $u\in\mathcal D(A)=\mathcal D(\Delta_{H^1})$ and $v\in H^1(S)$,
where we decompose $v=v^\top+v^\bot$ with $v^\top\in H^1$ and $v^\bot\in H^{1,\bot}$,
the orthogonal complement of $H^1$ in $H^1(S)$.

Now choose the closed subspace $\mathcal{H}_a$ of $H^1(S)$ as in the introduction,
and let $\Delta_a=\Delta_{\mathcal H_a}$ be the Friedrichs extension with respect to its closure $\mathcal{L}_a\subseteq L^2(S)$ as in the introduction.
Let $S_a$ the part of the hyperbolic surface $S$ below the horocycles in $H_a$; that is, $S_a$ is the complement of $C_a$.

\begin{lem}
If $u\in \mathcal D(\Delta_{a})$, then $u|_{S_a}\in H^2(S_a)$, $u|_{C_a}\in H^2(C_a)$, and
\[\Delta u|_{S_a}=(\Delta_{a}u)|_{S_a} \quad\text{and}\quad \Delta u|_{C_a}=(\Delta_{a}u)|_{C_a}.\]
\end{lem}

\begin{proof}
Clearly, $\Delta u|_{S_a}$ exists in $L^2({S_a})$ in the sense of distributions.
Hence $u|_{S_a}$ belongs to $H^2(S_a)$, by the ellipticity of $\Delta$, and $\Delta u|_{S_a}=(\Delta_{a}u)|_{S_a}$.

For $u|_{C_a}$, we separate two complementary classes of $v\in C^\infty_c(C_a)$.
The first class consists of functions $v$ which are constant along horocycles of $C_a$.
Then, by the constant mean curvature of the horocycles, $\Delta v$ is also constant along horocycles of $C_a$ and hence
\begin{align*}
    \la\nabla u,\nabla v\ra_0 = \la u,\Delta v\ra_0 = 0.
\end{align*}
The second class consists of functions $v$ which have vanishing averages over horocycles of $C_a$.
Then $v=v^\top\in\mathcal H_a$ and hence
\begin{align*}
    \la\Delta_{a}u,v\ra_0 = \la\nabla u,\nabla v\ra_0 = \la\nabla u|_{C_a},\nabla v\ra_0 = \la u|_{C_a},\Delta v\ra_0.
\end{align*}
In conclusion, $\Delta  u|_{C_a} =(\Delta_{a}u)|_{C_a}$ exists in the sense of distributions.
In particular, $u|_{C_a}$ belongs to $H^2(C_a)$.
\end{proof}

Let $u\in\mathcal H_a$, and suppose that $u|_{S_a}\in H^2(S_a)$ and $u|_{C_a}\in H^2(C_a)$.
Then
\begin{align}\label{h2}
\begin{split}
    \la u,v\ra_1
    &= \la u,v\ra_0 +\la\nabla u,\nabla v\ra_0 \\
    &= \la u,v\ra_0 + \la \Delta  u|_{S_a},v\ra_{S_a,0} + \la \Delta u|_{C_a},v\ra_{C_a,0} + T_uv \\
    &= \la u,v\ra_0 + \la \Delta u,v\ra_{0} + T_uv
\end{split}
\end{align}
for all $v\in H^1(S)$, where $\nu$ denotes the outward unit normal of $S_a$ along $H_a$
and $T_u\in H^{-1}(S)$ is given by
\begin{align}\label{tuv}
    T_uv=\int_{H_a} (\nabla_\nu u|_{S_a} - \nabla_\nu u|_{C_a}) ~ v.
\end{align}
Note here that, in this integral, the corresponding traces along $H_a$ exist and that the traces of $v|_{S_a}$ and $v|_{C_a}$ along $H_a$ coincide in $H^{1/2}(H_a)\subseteq L^2(H_a)$.
For the right hand side to be at most $C_u|v|_0$ for all $v\in\mathcal H_a$,
we need that $T_u$ vanishes on $\mathcal H_a$ and hence that $\nabla_\nu u|_{S_a} - \nabla_\nu u|_{C_a} = \text{\rm loc.\,const}$ along $H_a$.
Hence we arrive at

\begin{prop}\label{h2d}
The domain $D(\Delta_{a})$ of $\Delta_{a}$ consists of all $u\in\mathcal{H}_a$ such that $u|_{S_a}\in H^2(S_a)$, $u|_{C_a}\in H^2(C_a)$,
and such that, along $H_a$,
\begin{align*}
	u|_{S_a}=u|_{C_a} 
	\quad\text{and}\quad
	\nabla_\nu u|_{S_a} = \nabla_\nu u|_{C_a} + \text{\rm loc.\,const}.
\end{align*}
For any $u\in D(\Delta_{a})$, we have
\begin{align*}
    \Delta_{a}u = \Delta u|_{S_a} + \Delta u|_{C_a}.
\end{align*}
\end{prop}

A priori, $\Delta u|_{S_a}$ and $\Delta u|_{C_a}$ in the above formula are only defined on $S_a$ and $C_a$,
but we extend them by zero to the respective rest of $S$.

\begin{rem}
Extending by zero on respective rests of $S$ as above, we also have
\begin{align*}
	\mathcal D(\Delta_{a})
	= H^2(S_a) \oplus_\tau \big\{v\in H^2(C_a) \,\big| \int_{H_{i,t}} v=0 \,~ \text{for all $t \ge a_i$, $1\le i\le n$} \big\},
\end{align*}
where the transfer condition $\tau$ is
\begin{align*}
	\tau(u,v): \quad u|_{H_{a}} = v|_{H_{a}} \quad\text{and}\quad \nabla_\nu u &= \nabla_\nu v + \text{\rm loc.\,const}
\end{align*}
in $L^2(H_a)$; compare with (8.3) in \cite{LP}.
In contrast,
\begin{align*}
	H^2(S)
	= H^2(S_a) \oplus_\sigma H^2(C_a),
\end{align*}
where the transfer condition $\sigma$ is now stronger,
\begin{align*}
	\sigma(u,v): \quad
	u|_{H_a} = v|_{H_a}  \quad\text{and}\quad \nabla u = \nabla v
\end{align*}
in $L^2(H_a)$ and $L^2(H_a,TS|_{H_a})$, respectively.
\end{rem}

\begin{rem}
Letting $E=\{u\in H^1(S) \mid u|_{S_a}\in H^2(S_a), u|_{C_a}\in H^2(C_a)\}$,
we get from \eqref{h2} the weak Laplacian for $u\in E$,
\begin{align*}
	\la u,\Delta v\ra_0
	= \la \Delta u|_{S_a} + \Delta u|_{C_a},v\ra_0 + \int_{H_a} (\nabla_\nu u|_{S_a} - \nabla_\nu u|_{C_a})v
\end{align*}
for all $u\in E$ and $v\in C^\infty_c(X)$.
Hence $\nabla_\nu u|_{S_a} - \nabla_\nu u|_{C_a} = \text{\rm loc.\,const} =: \vf_u$ along $H_a$ for $u\in\mathcal D(\Delta_{a})$ just means  that $\Delta_{\mathrm{weak}}u-\vf_u T_a\in L^2(X)$, where $T_av$ is the locally constant function on $H_a$ which assigns to points on $H_{a_i}$  the Fourier coefficient of order zero of $v$ along $H_{a_i}$;
compare with \cite[Th\'eor\`eme 1]{CV}, where the case $n=1$ is treated.
\end{rem}

\cref{h2d} has the following

\begin{cor}\label{cinf}
If $u\in D(\Delta_a)$ is an eigenfunction of $\Delta_a$, then $u|_{S_a}$ and $u|_{C_a}$ are $C^\infty$ functions on $S_a$ and $C_a$ such that, along $H_a$,
\begin{align*}
    u|_{S_a} = u|_{C_a}
    \quad\text{and}\quad
    \nabla_\nu u|_{S_a} = \nabla_\nu u|_{C_a} + \text{\rm loc.\,const}.
\end{align*}
\end{cor}

\begin{proof}
Elliptic regularity implies that $u|_{S_a}\in H^k(S_a)$ and $u|_{C_a}\in H^k(C_a)$ for all $k$.
The assertion now follows from, e.g., \cite[Proposition 4.4.3]{T1}.
\end{proof} 

Obviously, the assertion of \cref{cinf} extends to finite linear combinations of eigenfunctions of $\Delta_a$.

\section{Spectral theory of the Laplacian}\label{speclap}
For convenience, we collect some facts on the spectral theory the Laplacian of $S$.
We assume throughout this section that $S=\Gamma\backslash\H$ is of finite area $|S|$ with $n\ge1$ cusps.
We also assume that $S$ is oriented since we are not aware of references covering the non-orientable case.

\subsection{Spectral theory on the standard hyperbolic cusp}
In the upper half space model $\H=\{y>0\}$ of the hyperbolic plane, the \emph{standard hyperbolic cusp} $C_0$ is the quotient of $\H$ by the infinite cyclic group of isometries generated by
\begin{align}\label{c0}
    \gamma_0 = \begin{pmatrix} 1 & 1 \\ 0 & 1\end{pmatrix},
\end{align}
where, here and below, we write $z=x+iy$ for elements of $\H$.
For any $y>0$, the circle $H_y=\{x+iy\mid x\in\Z\backslash\R\}$ in $C_0$ will be called the \emph{horocycle} at height $y$.
The length of $H_y$ is $1/y$.
Hence the functions
\begin{align}
    y^{1/2}e^{2\pi i kx}, \quad\text{for $k\in\Z$}, 
\end{align}
form a Hilbert basis of $L^2(H_y)$, for all $y>0$.
They define the \emph{Fourier coefficients} $f_k(y)$ of smooth functions $f$ on $C_0$, where
\begin{align}\label{fc}
    f_k(y) = \int_0^1 f(x+iy)y^{1/2}e^{2\pi i kx} dx,
\end{align}
which are smooth in $y>0$.
The Laplacian on $\H$ and $C_0$ is given by
\begin{align}\label{hl}
    \Delta &= -y^2(\partial_x^2+\partial_y^2)
\end{align}
If we write a smooth function $f$ in terms of its Fourier coefficients, then
\begin{align}\label{df}
\begin{split}
    \Delta f &= \Delta \sum\nolimits_{k\in\Z} f_ky^{1/2}e^{2\pi i kx} \\
     &= \sum\nolimits_{k\in\Z} \big(-y^2f_k'' - yf_k' + (1/4 + 4\pi^2k^2y^2)f_k\big)y^{1/2}e^{2\pi i kx},
\end{split}
\end{align}
which gives the Fourier decomposition of $\Delta f$.
Hence $f$ is a $\lambda$-eigenfunc\-tion of $\Delta$, for $\lambda\in\C$, if and only if the functions $f_k$ satisfy
\begin{align}\label{dfl}
    y^2f_k'' + yf_k' - (1/4 + 4\pi^2k^2y^2 - \lambda) = 0,
\end{align}
which is a family of linear ODEs of second order.

\begin{cor}\label{fha}
If $f$ is a $\lambda$-eigenfunction of $\Delta$, then $f$ is determined by $f|_{H_a}$ and $\partial_yf|_{H_a}$, for any given $a>0$.
\end{cor} 

We write $\lambda=s(1-s)\in\C$ and note in passing that $s(1-s)$ defines a double cover of the complex $\lambda$-plane with branch point at $s=1/2$.
For $k=0$, there are the following two fundamental solutions of \eqref{dfl},
\begin{align}\label{f0}
    f_0(y) = y^{s-1/2} \text{ and } y^{1/2-s}
    \quad\text{respectively}\quad
    1 \text{ and } \ln y
\end{align}
for $\lambda\ne1/4$ (that is, $s\ne1/2$) respectively $\lambda=1/4$ (that is, $s=1/2$). 
In particular, if $f$ only depends on $y$, then there are the following two fundamental solutions of $\Delta f=\lambda f$,
\begin{align}\label{fy}
    f(y) = y^{s} \text{ and } y^{1-s}
    \quad\text{respectively}\quad
    y^{1/2} \text{ and } y^{1/2}\ln y
\end{align}
for $\lambda\ne1/4$ (that is, $s\ne1/2$) respectively $\lambda=1/4$ (that is, $s=1/2$).
For $k\ne0$, we have the following two fundamental solutions of \eqref{dfl},
\begin{align}\label{fk}
    f_k(y) = 2K_{s-1/2}(2\pi ky) \quad\text{and}\quad f_k(y) = 2\pi I_{s-1/2}(2\pi ky),
\end{align}
where $I_\nu$ and $K_\nu$ denote the Bessel functions associated to the ODE \[y^2f''+yf'-(y^2+\nu^2)f=0;\] see \cite[Appendix B.4]{Iw}.

\subsection{Eisenstein series}\label{seceis}
We follow the exposition in \cite{Iw}.
Let $S=\Gamma\backslash\H$ be a hyperbolic surface of finite area with $n\ge1$ cusps $C_1,\dots,C_n$.
For each of these, choose an ideal point $\xi_i$ with stabilizer $\Gamma_i\cong\la\gamma_i\ra$ in $\Gamma$
and an isometry $\sigma_i$ of the hyperbolic plane with $\sigma_i\infty=\xi_i$ (in the upper half-space model) such that 
\begin{align}\label{cuspi}
  C_i = \Gamma_i\backslash\sigma_i\{y \geq b_i\}, \quad\text{where}\quad
  \gamma_i = \sigma_i\begin{pmatrix} 1 & 1 \\ 0 & 1\end{pmatrix}\sigma_i^{-1}.
\end{align}
Note that, for $\xi_i$ given, $\sigma_i$ is unique up to a translation $n_t(x+iy)=x+t+iy$ of $\H$.
For a function $\vf$ on $\R^+$, we obtain the \emph{weighted Eisenstein series}
\begin{align}\label{eisi}
    E_i(z,\vf) = \sum\nolimits_{\gamma\in\Gamma_i\backslash\Gamma} \vf(\Im\sigma_i^{-1}\gamma z)
\end{align}
if well defined, and then also as a function on $S$.
For a different choice $\xi_i'=g\xi_i$ of $\xi_i$, where $g\in\Gamma$, the stabilizer $\Gamma_i'$ in $\Gamma$ is $g\Gamma_i g^{-1}$, and we have
\begin{align*}
    \sum\nolimits_{\gamma\in\Gamma_i'\backslash\Gamma}\vf(\Im n_t^{-1}\sigma_i^{-1}g^{-1}\gamma z) = \sum\nolimits_{\gamma\in\Gamma_i\backslash\Gamma}\vf(\Im\sigma_i^{-1}\gamma z) = E_i(z,\vf)
\end{align*}
for the Eisenstein series with weight $\vf$, where $\sigma_i'=g\sigma_in_t$, since $n_t$ does not change imaginary parts.
For $s\in\C$ and $\vf(y)=y^s$, we write
\begin{align}\label{eisi2}
    E_i(z,s) = E_i(z,y^s) = \sum\nolimits_{\gamma\in\Gamma_i\backslash\Gamma} (\Im\sigma_i^{-1}\gamma z)^s.
\end{align}
Define $s(1-s)$-eigenfunctions $V_s$ and $W_s$ by
\begin{align}\label{whi}
\begin{split}
    V_s(z) = 2\pi y^{1/2} I_{2-1/2}(2\pi y) e^{2\pi ix}, \\
    W_s(z) = 2y^{1/2} K_{s-1/2}(2\pi y) e^{2\pi ix},
\end{split}
\end{align}
where $W$ is called the \emph{Whittaker function}; compare with \cite[Equations 1.26, 1.27, 1.36]{Iw}.
Notice the different asymptotics of $V_s$ and $W_s$,
\begin{align}\label{whiy}
\begin{split}
    V_s(z) \sim e^{2\pi ix}e^{2\pi y} \quad\text{and}\quad
    W_s(z) \sim e^{2\pi ix}e^{-2\pi y} \quad\text{as $y\to\infty$};
\end{split}
\end{align}
see \cite[Equations 1.37, 1.38]{Iw}.
By \cite[Equation 6.18]{Iw}, we have
\begin{align}\label{I6.18}
    E_i(\sigma_jz,s) = \delta_{ij}y^s + \vf_{ij}(s)y^{1-s} + \sum\nolimits_{k\ne0}\vf_{ij}(k,s)W_s(kz),
\end{align}
where the coefficients $\vf_{ij}(s)$ and $\vf_{ij}(k,s)$ are meromorphic functions of $s$.
The matrix $\Phi(s)=(\vf_{ij}(s))$ of coefficients is called the \emph{scattering matrix}.

\begin{rem}\label{Ipage91}
For $s\ne 1/2$, the term of order $0$, that is, the sum of the first two terms on the right in \eqref{I6.18}, does not vanish identically, hence neither $E_i(.,s)$.
On the other hand, $E_i(.,1/2)$ vanishes identically if and only if $\vf_{ii}(1/2)=-1$.
\end{rem}

The following are Theorems 6.6, 6.9, 6.10 and 6.11 in \cite{Iw} together with \cite[Th\'eor\`eme 4]{CV}.

\begin{thm}\label{I6.6}
The \emph{scattering matrix} $\Phi(s)=(\vf_{ij}(s))$ is Hermitian for real $s$ and unitary for $\Re s=1/2$.
\end{thm}

\begin{thm}\label{I6.9}
The coefficients $\vf_{ij}(s)$ are holomorphic in $\Re s\ge1/2$, except for a finite number of simple poles in the interval $(1/2,1]$.
If $s$ is such a pole, then it is also a pole of $\vf_{ii}(s)$ and $\res\vf_{ii}(s)>0$.
\end{thm}

\begin{thm}\label{I6.10}
The map $s\mapsto E_i(.,s)\in C^\infty(S)$ is holomorphic in $\{ s\in \mathbb{C}\mid \Re s\ge1/2 \}$, except for
the poles of $\vf_{ii}(s)$.
The residues \[\res E_i(z,s)=\lim_{s'\to s}(s'-s)E_i(z,s')\] at the poles of $\vf_{ii}(s)$ are non-vanishing square-integrable $\lambda$-eigenfunc\-tions on $S$, where $\lambda=s(1-s)$.
\end{thm}

Note that taking the residue kills the first term of the Fourier coefficient of order $0$ in \eqref{I6.18},
whereas the second becomes $\res\vf_{ij}(s)y^{1-s}$.
Since the poles have real part $>1/2$, the square-integrability of the residue $\res E_i(.,s)$ at poles $s$ follows.
Non-vanishing of $\res E_i(z,s)$ holds since $\res\vf_{ii}(s)>0$ for them.

\subsection{Spectral decomposition}\label{sublap}

Eigenvalues and eigenfunctions of the kind in \cref{I6.10} are called \emph{residual}.
As an example, $s=1$ is a pole of each Eisenstein series $E_i(z,s)$, and the residue $\res E_i(z,1)=1/|S|$, by \cite[Proposition 6.13]{Iw}.
Note that the residue at $s=1$ is constant and does not depend on $i$, whereas the multiplicity of the eigenvalue $\lambda=0$ is one.
In \cref{mulres} below, we characterize the multiplicity of residual eigenvalues in terms of the residue of the scattering matrix. 

If $\vf$ is smooth and with compact support, then $E_i(z,\vf)$ is a well defined bounded smooth function on $S$, called an \emph{incomplete Eisenstein series}.
The vector space generated by incomplete Eisenstein series is denoted by $\mathcal E'$.

Say that a smooth function $f$ on $S$ is \emph{cuspidal} if its Fourier coefficients of order $0$ along the cusps of $S$ vanish.
More precisely, for any $1\le i\le n$, the Fourier coefficient of order $0$ of $\tilde f\circ\sigma_i$ vanishes (completely, not only for $y>b_i$),
where $\tilde f$ denotes the lift of $f$ to the upper half-plane.
The vector space of bounded cuspidal functions is denoted by $\mathcal C'$.
It is orthogonal to $\mathcal E'$ in $L^2(S)$.
The corresponding closures ${\mathcal E}$ and ${\mathcal C}$ in $L^2(S)$ decompose $L^2(S)$,
\begin{align}\label{coe}
    L^2(S) = \mathcal C \oplus \mathcal E,
\end{align}
see \cite[(3.16)]{Iw}.
This decomposition is $L^2$-orthogonal and $\Delta$-invariant and decomposes the spectrum as follows (\cite[Theorems 4.7 and 7.3]{Iw}).

\begin{thm}\label{specs}
The spectrum of $\Delta$ on the \emph{cuspidal part} $\mathcal C$ is discrete.
The spectrum of $\Delta$ on $\mathcal E$ admits an $L^2$-orthogonal and $\Delta$-invariant decomposition
\begin{align*}
    \mathcal E = \mathcal R \oplus \mathcal E_1 \oplus\cdots\oplus \mathcal E_n,
\end{align*}
where the spectrum of $\Delta$ on
\begin{enumerate}
    \item \emph{the residual part }$\mathcal R$ consists of the finitely many residual eigenvalues and is contained in $[0,1/4)$;
    \item each of the $\mathcal E_i$ is absolutely continuous, covering $[1/4,\infty)$ with multiplicity one.
\end{enumerate}
\end{thm}

The spectrum of $\Delta$ on $\mathcal E_i$ is spanned by the $E_i(z,1/2+ir),r\in\R$, although the latter do not belong to $L^2(S)$ themselves; compare with \cite[Theorem 7.3]{Iw} for details.

We denote the eigenvalues of $\Delta$ on $\mathcal C$, counted with multiplicity, by
\begin{align*}
    0 < \lambda_1^{\cus} \le \lambda_2^{\cus} \le \lambda_3^{\cus} \le \dots
\end{align*}
We call them the \emph{cuspidal eigenvalues} with their \emph{cuspidal eigenfunctions}.
The point spectrum of $\Delta$ in $\mathcal E$ consists of the finite number of residual eigenvalues
\begin{align*}
    0 = \lambda_0^{\res} < \lambda_1^{\res} \le \lambda_2^{\res} \le \dots \le \lambda_k^{\res} < 1/4,
\end{align*}
where we note that $0$ is the residual eigenvalue associated to the pole $s=1$ \cite[Theorem 6.13]{Iw}.

\begin{rem}\label{specrem}
The spectrum of $\Delta$ in $[0,1/4)$ is discrete.
It consists of finitely many cuspidal and residual eigenvalues of finite multiplicity.
If $1/4$ is an eigenvalue, then it is cuspidal of finite multiplicity which is at most $|\chi(\bar{S})| - 1$ by Otal's result \cite{Ot}.
\end{rem}

We write non-negative real numbers as $\lambda=s(1-s)$, where $1/2\le s\le 1$ or $s\in\C$ satisfies $\Re s=1/2$ and $\Im s\ge0$.
For any $s\in(1/2,1]$, we let $P(s)$ consist of those $1\le i\le n$ such that $s$ is a pole of the Eisenstein series $E_i$ and set $Q(s)=\{1,\dots,n\}\setminus P(s)$.
We say that $s$ is \emph{regular} if $P(s)=\emptyset$ and that $s$ is \emph{singular} otherwise.
All but finitely many $s$ are regular, and singular ones belong to $(1/2,1]$, by \cref{I6.10}.
We say that $s$ is \emph{completely singular} if $Q(s)=\emptyset$.
By \cite[Theorem 6.13]{Iw}, $s=1$ is completely singular.
If $S$ has exactly one end, then $s$ is either regular or completely singular.

\begin{prop}\label{mulres}
The multiplicity of $\lambda=s(1-s)$ as a residual eigenvalue of $\Delta$ is equal to $\rk\res\Phi(s)$ of the scattering matrix $\Phi(s)$.
Moreover, $\rk\res\Phi(s)\ge1$ if $P(s)\ne\emptyset$. 
\end{prop}

For example, $\lambda=0$ is a residual eigenvalue of multiplicity one.
Hence the rank of $\res\Phi(1)$ equals one.

\begin{proof}
The space of residual $\lambda$-eigenfunctions consists of all linear combinations $f$ of the residues $\res E_i(.,s)$, $i\in P(s)$.
If 
\begin{align*}
    \sum\nolimits_{i\in P(s)}\alpha_i\res\vf_{ij}(s)=0
\end{align*}
for all $1\le j\le n$, then the Fourier coefficients of order zero of $f$ vanish along the cusps, hence they vanish anywhere.
At the same time, $f$ is perpendicular to $\mathcal C$.
Hence $f=0$, and the first claim follows.
As for the second, we note that $\res\Phi(s)\ne0$ if $P(s)\ne\emptyset$.
\end{proof}

\section{Spectral theory of pseudo-Laplacians}\label{subcdv}

In \cite[Th\'eor\`eme 5]{CV}, Colin de Verdi\`ere obtains a description of the spectrum of $\Delta_a$ in terms of cusp forms of $\Delta_\infty$ and Eisenstein series in the case where $S$ has exactly one cusp.
In this subsection, we extend his result to the case where the surface may have more than one cusp.
His result corresponds to the union of \cref{thmy} and Corollaries \ref{corxx} and \ref{coryy} below in the case of one cusp.
The situation for more than one cusp is more complicated, and many of our arguments are different from his.
We keep the above setup and let $C_{i,b_i}=\Gamma_i\backslash\sigma_i\{y \geq b_i\}$, for $1\le i\le n$.

As in the previous section, let the point spectrum of $\Delta$ in $\mathcal E$, consisting of the finite number of residual eigenvalues, be denoted by
\begin{align*}
    0 = \lambda_0^{\res} < \lambda_1^{\res} \le \lambda_2^{\res} \le \dots \le \lambda_k^{\res} < 1/4.
\end{align*} 
Since $\Delta$ and $\Delta_a$ agree on $\mathcal C$, cuspidal eigenvalues are also eigenvalues of all $\Delta_a$.
For $a>b$, let
\begin{align*}
    0 < \lambda_0^a \le \lambda_1^a \le \lambda_2^a \le \dots
\end{align*}
be the eigenvalues of $\Delta_a$ on the orthogonal complement $\mathcal{E}_a=\mathcal{H}_a\cap\mathcal{E}$ of $\mathcal{C}$.

\begin{thm}\label{thmy}
For all $j\ge0$, $\lambda_j^a$ is continuous in $a$ for $b<a<\infty$.
Moreover, for all $b<a<a'$, we have
\begin{alignat}{2}
    \lambda_j^{\res}<\lambda_j^{a'}<\lambda_j^a &\quad\text{and}\quad \lim\nolimits_{a\to\infty}\lambda_j^a=\lambda_j^{\res} &\quad\text{for $j\le k$}; \tag{1} \label{y1} \\
    1/4<\lambda_j^{a'}<\lambda_j^a &\quad\text{and}\quad \lim\nolimits_{a\to\infty}\lambda_j^a=1/4 &\quad\text{for $j>k$}. \tag{2} \label{y2}
\end{alignat}
\end{thm}

\begin{proof}
Let $\vf$ be a smooth function with compact support in $(-2\delta,2\delta)$ such that $\vf=1$ on $(-\delta,\delta)$, where $\delta>0$ is sufficiently small.
For $1\le i\le n$, let $Y_i$ be the vector field on $S$, which vanishes outside the cusp $C_{i,b_i}$, is perpendicular to the horocycles in $C_{i,b_i}$, points away from infinity in $C_{i,b_i}$, and has length $\vf(\ln (y/a_i))$, where we use $\sigma_i^{-1}=x+iy$ as coordinates.
Note that the supports of the $Y_i$ are compact, hence they are complete.

Let $F_{i,s}$ be the flow of $Y_i$.
Then $F_{i,s}$ is equal to the identity outside the $2\delta$-neighborhood of $H_{i,a_i}$ and $F_{i,\ln(a_i'/a_i)}(C_{a_i'})=C_{a_i}$ for $|\ln(a_i'/a_i)|<\delta$.
Furthermore, since $Y$ is perpendicular to horocycles in $C_{i,b_i}$ and of constant length along them, $F_{i,s}$ maps horocycles to horocycles such that pull-back $F_{i,s}^*$ preserves the condition on functions that their Fourier coefficients of order zero along the horocycles vanish.
Hence pull-back $F_{aa'}^*$ with the concatenation $F_{aa'}$ of the $F_{i,\ln(a_i'/a_i)}$ maps $\mathcal{H}_{a'}$ to $\mathcal{H}_{a}$.

Since the $F_{i,s}$ tend to the identity of $S$ in the $C^\infty$-topology as $s$ tends to zero, the first conclusion follows from the variational characterisation of eigenvalues and the min-max principle, which characterizes $\lambda_j^{a}$ as the minimal number $\lambda$ such that, for some $(j+1)$-dimensional subspace $V\subseteq\mathcal{E}_a$,
\begin{align}\label{V}
    \|\nabla u\|_{L^2}^2\le\lambda\|u\|_{L^2}^2 \quad\text{for all $u\in V$.}
\end{align}
Note that $F_{aa'}^*$ does not map $\mathcal{E}_{a'}$ to $\mathcal{E}_a$ since $F_{aa'}^*$ does not preserve volume.
However, since $F_{aa'}$ tends to the identity of $S$ in the $C^\infty$-topology as $a'$ approaches $a$, $F_{aa'}^*(\mathcal{E}_{a'})$ approaches $\mathcal{E}_a$ and Rayleigh quotients approach each other under $F_{aa'}^*$ as well.
The assertion about continuity follows.

By \cref{specs} and the variational characterisation of eigenvalues, the eigenvalue of $\Delta$ on $\mathcal{E}$ corresponding to $\lambda_j(a)$ is obtained by replacing $\mathcal{E}_a$ by $\mathcal{E}$ in the requirement on $V$ in \eqref{V}, hence is equal to $\lambda_j^{\res}$ for $j\le k$ and $1/4$ for $j>k$.
Since $\mathcal{E}_a\subseteq\mathcal{E}_{a'}\subseteq\mathcal{E}$, we conclude the asserted inequalities in weak form.
Now an $\mathcal{E}_a$-eigenfunction of $\Delta_a$ cannot be an $\mathcal{E}_{a'}$-eigenfunction of $\Delta_{a'}$, since otherwise it would have vanishing first Fourier coefficient on the $[a,a']$-part of the cusps, hence on all of the cusps by analyticity on the interior part with respect to $a'$ and would then belong to $\mathcal{C}$.
But since it is orthogonal to $\mathcal{C}$, it would vanish.
This proves the strict inequalities $\lambda_j^{a'}<\lambda_j^a$,
and the proof of the strictness of the other inequalities follows from \cref{specs}.
The limiting behaviour is also a consequence of \cref{specs} since the $\mathcal{E}_a$ exhaust $\mathcal{E}$.
\end{proof}

\begin{cor}\label{cory}
For all $b<a<a'$, we have \[N_a(1/4)\le N_{a'}(1/4)\le N(1/4)\le|\chi(S)|\]
and $N_a(1/4)=N(1/4)$ for all sufficiently large $a$.
\end{cor}

\begin{proof}
By \cite[Theorem 1.5]{BMM}, $\Delta=\Delta_\infty$ has at most $|\chi(S)|$ eigenvalues in $[0,1/4]$, including multiplicity.
Therefore the assertion follows immediately from the monotonicity properties stated in \cref{thmy}.
\end{proof}

Consider now a linear combination of Eisenstein series and their residues at a given $s$,
\begin{align}\label{eisres}
    f_s = \sum\nolimits_{i\in Q(s)}\alpha_iE_i(z,s) + \sum\nolimits_{i\in P(s)}\alpha_i\res E_i(z,s).
\end{align}
Their Fourier coefficients of order $0$ along $C_j$ are
\begin{align}\label{eiss}
    \sum\nolimits_{i\in Q(s)}\alpha_i\vf_{ij}(s)y^{1-s} + \sum\nolimits_{i\in P(s)}\alpha_i\res\vf_{ij}(s)y^{1-s} +
    \begin{cases}
     \alpha_j y^s, \\
     \phantom{-}0,
    \end{cases}
\end{align}
for all $j\in Q(s)$ respectively $j\in P(s)$; compare with \eqref{I6.18}.

\begin{thm}\label{thmxx}
For any $s\ne1/2$ as above and $\lambda=s(1-s)$, the $\lambda$-eigenspace of $\Delta_a$ on $\mathcal{E}_a$ consists of the truncation at $a$ of the $f_s$ as in \eqref{eisres} such that, for all $1\le j\le n$,
\begin{align*}
     \sum\nolimits_{i\in Q(s)}\alpha_i\vf_{ij}(s)  + \sum\nolimits_{i\in P(s)}\alpha_i\res\vf_{ij}(s) =
     \begin{cases}
      -\alpha_ja_j^{2s-1} &\text{if $j\in Q(s)$},\\
      \phantom{-\alpha_j}0 &\text{if $j\in P(s)$}.
     \end{cases}
\end{align*}
If $f_s\ne0$, then there is $j\in Q(s)$ with $\alpha_j\ne0$.
\end{thm}

All but finitely many $s$ are regular, and for them the second sum on the left and the second case on the right are absent.
In this case, the displayed family of equations requires that $-1$ is an eigenvalue of the matrix with entries $\vf_{ij}(s)a_j^{1-2s}$, where $1\le i,j\le n$;
that is, the scattering matrix $\Phi(s)$, multiplied with the vector of $a_j^{1-2s}$.

\begin{cor}[Poles as barriers]\label{corxx}
If $s\in(1/2,1]$ is completely singular, then $\lambda$ is not an $\mathcal{E}_a$-eigenvalue of $\Delta_a$.
In particular, if $\ell$ denotes the number of residual eigenvalues below $\lambda$, counted with multiplicity, then
\begin{align*}
    \lambda_{\ell-1}^a < \lambda_{\ell}^{\res}=\lambda=\lambda_{\ell+r-1}^{\res} < \lambda_{\ell}^a
\end{align*}
for all $b<a<\infty$, where $1\le r=\rk\res\Phi(s)\le n$.
\end{cor}

\begin{proof}
Since $s$ is completely singular, $Q(s)=\emptyset$.
\end{proof}

In the case of one cusp, either $s$ is regular or completely singular.
Therefore we obtain the following assertions of \cite[Th\'e\`oreme 5]{CV}.

\begin{cor}\label{corxx1}
If $S$ has exactly one cusp, then $\mathcal{E}_a$-eigenvalues of $\Delta_a$ are of multiplicity one and not residual.
More precisely,
\begin{align*}
    0 = \lambda_0^{\res} < \lambda_0^a < \lambda_1^{\res} < \dots < \lambda_k^{\res} < \lambda_k^a.
\end{align*} 
\end{cor}

\begin{rem}
We see that in the case of exactly one cusp, the $\mathcal{E}_a$-eigenvalues $\lambda_i^{a}$ are contained in the intervals $(\lambda_i^{\res},\lambda_{i+1}^{\res})$ for $i<k$ and $(\lambda_k^{\res},\infty)$, respectively.
Since $\lambda_k^{\res}<1/4$, only $\lambda_k^a$ may pass through $1/4$ as $a$ varies.
For more than one cusp, $\mathcal{E}_a$-eigenvalues might pass through residual eigenvalues $\lambda=s(1-s)$ if $s$ is not completely singular.
\end{rem}

\begin{proof}[Proof of \cref{thmxx}]
The displayed family of equations imply that the conditions of \cref{cinf} are satisfied, since the Fourier coefficients of order $0$ of $f_s$ in the cusp $C_j$ vanish at $a_j$ and since the truncation of $f_s$ consists of deleting the Fourier coefficients of order $0$ beyond $a_j$, for all $1\le j\le n$.
Square-integrability of truncated Eisenstein series $E_i(\sigma_jz,s)$ follows from \cite[(6.20)]{Iw}, which says that
\begin{align*}
    E_i(\sigma_jz,s) = \delta_{ij}y^s + \vf_{ij}(s)y^{1-s} + O(e^{-2\pi y})
\end{align*}
if $s$ is not a pole of $E_i$.
Moreover, the residues of Eisenstein series, and hence their truncations, are also square-integrable.
Thus all the $f_s$ belong to the $\lambda$-eigenspace of $\Delta_a$ on $\mathcal{E}_a$.

To show that the $f_s$ span the space, suppose first that $s$ is regular.
Let $g$ be a $\lambda$-eigenfunction of $\Delta_a$ on $\mathcal{E}_a$.
Then the Fourier coefficient of $g$ of order $0$ on $[b_j,a_j]$ is equal to $\beta_jy^s+\gamma_jy^{1-s}$ for some constants $\beta_j,\gamma_j$, for all $1\le j\le n$.
These solutions to $\Delta f=\lambda f$ extend to all of the $[b_j,\infty)$.
Added to $g$ on the $(a_j,\infty)$, they yield a $\lambda$-eigenfunction $\hat g$ such that $g$ equals the truncation of $\hat g$ at $a$.
Write $s=1/2+it$ with $t<0$ or $t=i\sigma$ with $\sigma<0$.
Then the Fourier coefficient of $\hat g$ of order zero along the $j$-th cusp of $S$ is of the form
\begin{align*}
    \beta_{j}y^s + \gamma_{j}y^{1-s} = y^{1/2}(\beta_{j}y^{it} + \gamma_{j}y^{-it}).
\end{align*}
Therefore the Fourier coefficient of order zero along the $j$-th cusp of
\begin{align*}
    h = g - \sum \beta_j E_j(.,s) 
\end{align*}
does not involve $y^s$, but only $y^{1-s}$ by \eqref{I6.18}.
Hence $f=(h,ith)$ satisfies the assumptions of \cite[Theorem 8.4]{LP} (see line 1 on \cite[page 203]{LP}), and hence $f=0$; hence also $h=0$.

Suppose now, more generally, that $s\ne1/2$, let $\lambda=s(1-s)$, and recall that the cuspidal spectrum of $\Delta$ is discrete and belongs to the spectrum of all $\Delta_{a'}$.
Moreover, the $\mathcal{E}_a$-spectrum of $\Delta_a$ is also discrete, depends continuously on $a$, and is strictly monotonic.
Hence there is an $\ve>0$ such that the number of $\mathcal{E}_{a'}$-eigenvalues of $\Delta_{a'}$ between $\lambda-\ve$ and $\lambda+\ve$, counted with multiplicity, is equal to the multiplicity $\mu$ of $\lambda$ as a $\mathcal{E}_a$-eigenvalue of $\Delta_a$, for all $a'$ sufficiently close to $a$.
By monotonicity, $\Delta_{a'}$, for all $a'>a$ sufficiently close to $a$, has exactly $\mu$ $\mathcal{E}_{a'}$-eigenvalues in $(\lambda-\ve,\lambda)$.

Now consider a sequence of $a_n>a$ converging to $a$ and $\mathcal{E}_{a_n}$-eigenfunctions $f_{s_n}$ with $s_n\to s$ such that $f_{s_n}$ converges to an $\mathcal{E}_a$-eigenfunction $f_s$.
Write $f_{s_n}$ as in \eqref{eisres}, observing that $P(s_n)=\emptyset$.
For $i\in P(s)$, divide the $\alpha_{in}$ by $s_n-s$ and multiply the $E_i(.,s_n)$ by $s_n-s$ to obtain that the limit is of the form as in \eqref{eisres}.
Moreover, the Fourier coefficient of order zero of the limit $f_s$ vanishes at $a$, hence it satisfies the asserted equations in \cref{thmxx}.
Such limits span the whole $\lambda$-eigenspace of $\Delta_a$ on $\mathcal{E}_a$, by a dimension counting argument, and so the first part of the theorem follows.

As for the last assertion, if $\alpha_j=0$ for all $j\in Q(s)$, then the first sum on the left and the first case on the right in the displayed equation of \cref{thmxx} are absent.
Hence the displayed equation requires $(\alpha_1,\dots,\alpha_n)\in\ker\res\Phi(s)$.
Since this condition does not depend on the choice of $a$, the Fourier coefficients of $f_s$ of order zero vanish along all of the cusps, hence anywhere.
Therefore $f_s$ belongs to $\mathcal C$.
At the same time, $f_s$ is perpendicular to $\mathcal{C}$, hence $f_s=0$.
\end{proof}

Following Wolpert \cite[Definition 3.1]{Wo}, we say that a function $f$ on $S$ has moderate growth if there exists a constant $c>0$ such that, for each cusp $C_{i,b_i}$ of $S$, $f(\sigma_iz)=O(y^c)$ as $y\to\infty$, where we write $z=x+iy$ as usual.
As in  \cite[Definition 4.1]{Wo}, we let $\mathcal{E}(1/4)$ denote the space of generalised $1/4$-eigenfunctions of $\Delta$ of moderate growth perpendicular to $\mathcal C$.

The scattering  matrix $\Phi(1/2)=(\vf_{ij}(1/2))$ is Hermitian and unitary, hence $\C^n$ splits orthogonally as the sum of its eigenspaces $E^{\pm}$ for the eigenvalues $\pm1$ with multiplicities $\mu_{\pm}\ge0$.
Set
\begin{align*}
    \mathcal{F}^+ &= \{ \sum\nolimits_{1\le i\le n} \alpha^+_i E_i(z, 1/2) \mid \alpha^+ \in E^+ \},\\
    \mathcal{F}^- &= \{ \sum\nolimits_{1\le i\le n} \alpha^-_i \partial_sE_i(z, 1/2) \mid \alpha^- \in E^- \},
\end{align*}
where $\alpha^\pm=(\alpha_1^\pm,\dots,\alpha_n^\pm)\in E^\pm$.
A word about the definition of $\mathcal{F}^-$: any linear combination
\begin{align*}
    f = \sum\nolimits_i \alpha^-_iE_i(z, 1/2) = 0.
\end{align*}
As for the proof of the latter, since $\alpha^-\in E^-$, the zeroth order Fourier coefficient of $f$ vanishes identically along each cusp of $S$.
Therefore, since $f$ is perpendicular to $\mathcal C$, it must be identically zero, by \cref{specs}.

\begin{thm}\label{thmyy}
We have $\mathcal{E}(1/4) = \mathcal{F}^+ \oplus \mathcal{F}^-$, i.e., any generalised $1/4$-eigenfunction of $\Delta$ of moderate growth perpendicular to $\mathcal C$ is given by a unique linear combination
\begin{align}\label{eispar}
    f_s = f_{1/2} = \sum\nolimits_i\big(\alpha_i^+ E_i(z,1/2) + \alpha_i^-\partial_s E_i(z,1/2)\big).
\end{align}
\end{thm}

\begin{proof}
It follows easily from their definitions that the functions $E_i(z, 1/2)$ and $\partial_s E_i(z, 1/2)$ belong to $\mathcal{E}(1/4)$, for $1\le i\le n$.
Clearly, many of these can be linearly dependent.
Indeed, by \cite[Lemma 4.5]{Wo}, $\mathcal{E}(1/4)$ has dimension at most $n$.

The zeroth Fourier coefficient of $f=\sum\nolimits_i \alpha^+_i E_i(z, 1/2)\in\mathcal{F}^+$ along $C_{j,b_j}$ is
\[
  \big(\alpha_j^+ + \sum\nolimits_i \alpha^+_i\phi_{ij}(1/2) \big) y^{1/2} = 2\alpha_j^+y^{1/2},
\]
where we use $\alpha^+\in E^+$ for the latter equation.
Hence $f\ne0$ if $\alpha^+\ne0$.
Similarly, the zeroth Fourier coefficient of $f=\sum_i\alpha_i^- \partial_s E_i(z, 1/2)\in\mathcal{F}^-$ along $C_{j,b_j}$ is
\begin{align*}
    &\big(\alpha_j^- - \sum\nolimits_i \alpha^-_i\phi_{ij}(1/2) \big) y^{1/2}\ln y + \sum\nolimits_i \alpha^-_i\phi_{ij}'(1/2)y^{1/2} \\
    &= 2\alpha_j^- y^{1/2}\ln y + \sum\nolimits_i \alpha^-_i\phi_{ij}'(1/2)y^{1/2}.
\end{align*}
Hence $f\ne0$ if $\alpha^-\ne0$.
It follows that $\mathcal{F}^\pm$ is of dimension $\dim E^\pm$.
It also follows that $\mathcal{F}^+\cap\mathcal{F}^-=\{0\}$ since the zeroth Fourier coefficients have different growths along the cusps of $S$.
Now $\dim E^++\dim E^-=n$, and hence $\mathcal{E}(1/4) = \mathcal{F}^+ \oplus \mathcal{F}^-$ as asserted.
\end{proof}

As an aside, we note that the above proof shows that $\dim\mathcal{F}^\pm=\dim E^\pm$ and that $\dim\mathcal{E}(1/4)=n$.
The latter equation was also obtained in \cite[Remark 4.6]{Wo}.

\begin{thm}\label{proyy}
For $a > b$, any $1/4$-eigenfunction of $\Delta_a$ in $\mathcal{E}_a$ consists of the truncations at $a$ of all linear combinations $f_s$ as in \eqref{eispar} with vanishing zeroth Fourier coefficient at $a$; that is, such that, for all $1\le j\le n$,
\begin{align*}
   2\alpha_j^+ + 2\alpha_j^-\ln a_j + \sum\nolimits_{1\le i\le n}\alpha_i^-\vf_{ij}'(1/2) = 0.
\end{align*}
\end{thm}

\begin{proof}
Following the strategy in the proof of \cref{thmxx}, we can extend any $1/4$-eigenfunction of $\Delta_a$ in $\mathcal{E}_a$ to all of $S$ analytically.
Since any eigenfunction of $\Delta_a$ is in $L^2(S)$, the extended function is of moderate growth, and hence has the form \eqref{eispar}.
The last identity is obtained by considering the zeroth Fourier coefficient of the extended function in each cusp and equating that to zero at $y = a_i$ at the $i$-th cusp.
\end{proof}

\begin{cor}\label{coryy}
The multiplicity of $1/4$ as an eigenvalue of $\Delta_a$ in $\mathcal{E}_a$ equals
\begin{align*}
    \mu = \dim\{\alpha\in E^- \mid (D_a+\Phi'(1/2)^{t})\alpha\in E^+\},
\end{align*}
where $D_a$ denotes the diagonal matrix with entries $2\ln a_j$.
In particular, $\mu\le\dim E^-$ and, if $\Phi(1/2)=-I$, then $\mu=\dim\ker(D_a+\Phi'(1/2)^{t})$.
\end{cor}

\begin{rem}\label{remps}
Phillips and Sarnak conjectured that $\Phi(1/2) = -I$ for a generic hyperbolic metric of finite area on a surface of finite topological type; see \cite[page 28]{PS}. 
\end{rem}

\begin{cor}\label{corxy}
Let $f\in L^2(S)$ be a $\lambda$-eigenfunction of $\Delta_a$ on $\mathcal{E}_a$.
Then $f$ is the truncation of a unique $\lambda$-eigenfunction $\hat f$ of $\Delta$ on $\mathcal{E}$,
and $\hat f$ is of moderate growth, but not square-integrable.
\end{cor}

\subsection{The non-orientable case}\label{subnono}

In \cref{speclap} and the above part of \cref{subcdv}, we assume throughout that $S$ is orientable.
Now we explain shortly, how the results relevant for the proof of \cref{mainq} and the remaining sections follow in the non-orientable case.

To that end, let $S$ be a non-compact non-orientable hyperbolic surface of finite area and $p\colon S'\to S$ be the orientation covering with covering transformation $\tau$ of $S'$, where $S'$ is endowed with the lifted metric.
Then $S$ is diffeomorphic to a closed non-orientable surface with a finite number of punctures $p_1,\dots,p_n$.
As in the orientable case, we can choose pairwise disjoint closed cusps $C_{i,b_i}$ around the $p_i$ with horocycles $H_{i,b_i}$ as boundaries.
Since cusps are orientable, each $C_{i,b_i}$ lifts to two isometric cusps $C_{i,b_i}'$ and $C_{i,b_i}''$ in $S'$, interchanged by $\tau$.

Keeping the notation as before, we obtain, for any $a>b$, a pseudo-Laplacian $\Delta_a$ on $S$.
The pull-back with $p$ intertwines $\Delta_a$ with the corresponding pseudo-Laplacian $\Delta_a'$ on $S'$,
where here $\Delta_a'$ is associated to the choices $a_i$ for $C_{i,a_i}'$ and $C_{i,a_i}''$.
It follows that eigenfunctions of $\Delta_a$ lift to eigenfunctions of $\Delta_a'$ and that the spectrum of $\Delta_a$ is discrete.
This is sufficient for our discussion in the following sections.

Now \cref{mainq} follows immediately from \cref{thmy} in the orientable case; hence a version of \cref{thmy} in the non-orientable case is sufficient for our purposes.
Since the pull-back of $\mathcal{C}$ equals $\mathcal{C}'$,
the orthogonal complement $\mathcal{E}$ of $\mathcal{C}$ in $L^2(S)$ pulls back to $\mathcal{E}'$
and $\Delta_a$ leaves $\mathcal{C}$ and $\mathcal{E}$ invariant.
The pull back of eigenfunctions of $\Delta_a$ consists of $\tau$-invariant eigenfunctions on $S'$.
Now going through the proof, we see that \cref{thmy} holds true in the non-orientable case as well.

It is also noteworthy that \cref{corxy} extends to the non-orientable case. Indeed, if $f \in L^2(S)$ is a $\lambda$-eigenfunction of $\Delta_a$ on $\mathcal{E}_a$, then its lift $f^\prime$ to $S^\prime$ is a $\lambda$-eigenfunction of $\Delta_a^\prime$ on $\mathcal{E}_a^\prime$. Therefore, $f^\prime$ is the truncation of a $\lambda$-eigenfunction $\hat{f}$ of $\Delta^\prime$ on $\mathcal{E}^\prime$. The assertion follows after noticing that $\hat{f}$ is $\tau$-invariant, since $\hat{f} - \hat{f} \circ \tau$ is a $\lambda$-eigenfunction of $\Delta^\prime$ vanishing in $S_a^\prime$.

\section{Analytic systoles}\label{secana}

Now we bring an adapted version of the analytic systole \cite{BMM} into play.
Note first that the Rayleigh quotient of any non-vanishing $u\in D(\Delta_a)$ is given by
\begin{align*}
    \Ray(u) = \frac{\la\Delta_au,u\ra_0}{\|u\|_0^2} = \frac{\|\nabla u\|_0^2}{\|u\|_0^2}.
\end{align*}
If $(\Omega_i)_i$ denotes the family of nodal domains of $u$, then we obtain
\begin{align*}
    \Ray(u) = \frac{\sum_i\int_{\Omega_i}|\nabla u|^2}{\sum_i\int_{\Omega_i}|u|^2}
\end{align*}
For any domain $\Omega\subseteq S$, let
\begin{align*}
    \lambda_0(\Omega) = \inf_{u\in\Lip_0(\Omega)}\Ray(u).
\end{align*}
The infimum on the right is also achieved by letting $u$ run over smooth functions with compact support in $\Omega$.

Recall from \cite{BMM} the concept of \emph{analytic systole} of $S$,
\begin{align*}
    \Lambda(S) = \inf_\Omega\lambda_0(\Omega),
\end{align*}
where $\Omega$ runs over all domains in $S$ with piecewise smooth boundary which are homeomorphic to closed discs or annuli or M\"obius bands.
Since we take an infimum, we can assume without loss of generality that the competing domains are compact with smooth boundary.
Because of the cusps, we have $\Lambda(S)=1/4$, which is not good enough for our purposes, and we need to refine the notion.

We distinguish the following types of compact domains $\Omega$ with piecewise smooth boundary in $S$:

Type 1: $\Omega$ is an embedded disc.
Then $\Omega$ lifts isometrically to a disc $\tilde\Omega$ in the hyperbolic plane and, hence, by the Faber-Krahn inequality, we have
\begin{align}\label{disc}
    \lambda_0(\Omega) = \lambda_0(\tilde\Omega) \ge \lambda_0(D_{|\Omega|}) \ge \lambda_0(D_{|S|}),
\end{align}
where $D_V$ denotes a disc in the hyperbolic plane of area $V$.

Type 2: $\Omega$ is an embedded annulus, whose boundary circles are homotopically trivial.
Then a connected component of the complement of $\Omega$ in $S$ is a disc $B'$ in $S$ which we glue to $\Omega$ to obtain a disc $B=\Omega\cup B'$.
We obtain
\begin{align}\label{aht}
    \lambda_0(\Omega) \ge \lambda_0(B) \ge \lambda_0(D_{|B|}) \ge \lambda_0(D_{|S|}).
\end{align}

Type 3: $\Omega$ is an embedded annulus, whose boundary circles are homotopic to a simple closed geodesic $c$ in $S$.

Type 4: $\Omega$ is an embedded M\"obius band.
Then the boundary circle of $\Omega$ is homotopic to a closed geodesic $c$ in $S$.

Type 5:  $\Omega$ is an annulus, whose boundary circles are homotopic to a horocycle in (exactly) one of the cusps of $S$
such that $\Omega$ reflects the requirement that functions in $\mathcal{H}_a$ have vanishing first Fourier coefficients beyond $H_a$.
That is, we assume further that, for each horocycle $H_y$ above $H_a$, if the interior of such an annulus intersects $H_y$, then its boundary meets $H_y$ in at least two points.

Finally, we define the \emph{analytic systole} of $S$ with respect to $\Delta_a$ to be
\begin{align*}
    \Lambda_a(S) = \inf_\Omega \lambda_0(\Omega)
\end{align*}
where $\Omega$ runs over all domains of the above five types.

\begin{prop}\label{prop:analytic-systole}
For any $a>b$, we have $\Lambda_a(S) > 1/4$. Furthermore, for any domain $\Omega$ of one of the above types, we have $\lambda_0(\Omega)>\Lambda_a(S)$.
\end{prop}

\begin{proof}
To prove the first assertion, we can consider the five types of domains as above separately.
For $\Omega$ of Type 1 and 2, $\lambda_0(\Omega)$ is bounded from below by $\lambda_0(D_{|S|})$, by the Faber-Krahn inequality.
Now by monotonicity, the latter is strictly larger than $\lambda_0$ of  the hyperbolic plane, which is $1/4$.
For $\Omega$ of Type 3 and 4, $\lambda_0(\Omega)$ is bounded from below by
\begin{align*}
    \frac14 + \min\bigg\{\frac{\pi}{|S|},\frac{\sys(S)^2}{|S|^2}\bigg\},
\end{align*}
where $\sys(S)$, the \emph{systole of $S$}, is the minimal length of simple closed geodesics in $S$; see \cite{M1} and \cite[Theorem 4.1(1)]{BMM2}.
Since $\sys(S)>0$, the above term is strictly larger than $1/4$.

It remains to consider domains of Type 5.
Let $\Omega_n$ be a competing sequence of domains in $S$ of this type.
Up to extracting a subsequence, we may assume that the boundary circles of $\partial\Omega_n$ are homotopic to $c=H_{i,a_i}$,
a generating loop $c$ of the fundamental group $\Gamma_i$ of a fixed cusp $C=C_{i,a_i}$ of $S$.
We order the boundary circles as $\partial\Omega_n=c_n\cup c_n^+$, where $c_n$ is below $c_n^+$ with respect to $C$.
From now on, by passing to larger annuli if necessary, we assume that the $c_n^+$ are inside $C$ and the $c_n$ do not intersect $C$.

Assume now to the contrary of the assertion that $\lambda_0(\Omega_n) \rightarrow \lambda \leq 1/4$, and consider $\vf_n \in C^\infty_c(S)$ supported in $\Omega_n$ with $\| \vf_n \|_{L^2} = 1$ and $\Ray\vf_n \rightarrow \lambda$. Keeping in mind that $\vf_n$ vanishes at some point on each horocycle $H$ above $H_{i,a_i}$, we derive from Lemma \ref{cusp estimate} that
\begin{equation}\label{decay}
    \int_{C_{i,y}} |\nabla \vf_n|^2 \geq \pi^2 y^2 \int_{C_{i,y}} \vf_n^2
\end{equation}
for any $y\geq a_i$.
Let $C_n=C\cup\Omega_n$, the domain of $S$ with boundary $c_n$ that contains $C$.
Choose a point $x_0 \in c =\partial C$ and observe that there exists $R>0$ such that $d(x_0,c_n)\leq R$ for any $n\in \N$.
Indeed, since $c_n$ is homotopic to $c$, it must intersect $S_a$ which has a finite diameter.
It should also be noticed that the image of any minimizing geodesic from $x_0$ to $c_n$ is contained in $C_n$.

Let $\Gamma_i$ denote the parabolic subgroup of $\pi_1(S,x_0)$ that corresponds to $c$, now viewed as acting on the hyperbolic plane. 
We consider the bi-infinite cusp $\mathcal{C} = \Gamma_i\backslash\H$.
The domains $C_n$ correspond to domains $\tilde{C}_n$ in $\mathcal{C}$.
Clearly, there exists $\bar{a}\le a_i$ independent of $n$ such that $\partial \tilde{C}_n \cap \mathcal{C}_{\bar{a}} \neq \emptyset$, where $\mathcal{C}_{\bar{a}} = \Gamma \backslash \{ x + i y : y \geq \bar{a} \}$, and such that the shortest geodesic joining $\tilde{x}_0$ and $\partial \tilde{C}_n$ is contained in $\tilde{C}_n$.

Denote by $\tilde{\Omega}_n \subset \tilde{C}_n$ the lifted annuli and consider the corresponding functions $\tilde{\vf}_n \in C^\infty_c(\mathcal{C})$ supported in $\tilde{\Omega}_n$.
Keeping in mind that $\lambda_0(\mathcal{C}) = 1/4$ is not an eigenvalue of the Friedrichs extension of the Laplacian on $\mathcal{C}$, we deduce from \cite[Proposition 6.2]{P} that $\lambda = 1/4$ and $\tilde{\vf}_n \rightharpoonup 0$ in $L^2(\mathcal{C})$.
The latter implies that $\tilde{\vf}_n \rightarrow 0$ in $L^2(K)$ for any compact subset $K$ of $\mathcal{C}$.
Moreover, we derive from \ref{decay} that
\[
\int_{\mathcal{C}_y} |\nabla \tilde{\vf}_n|^2 \geq \pi^2 y^2 \int_{\mathcal{C}_y} \tilde{\vf}_n^2
\]
for any $y\geq a_i$.
Since the square of the $L^2$-norm of $\nabla\tilde\vf$ is bounded by $1+1/4$ for all sufficiently large $n$,
we get that $\tilde{\vf}_n \rightarrow 0$ in $L^2(\mathcal{C}_y)$ for any $y>0$.
Indeed, given any $y>0$ and $\ve>0$, there exists $y^\prime\ge\max\{y,a_i\}$ such that $\| \tilde{\vf}_n \|_{L^2(\mathcal{C}_{y^\prime})} < \ve$ for any $n$ sufficiently large, and $\tilde{\vf}_n \rightarrow 0$ in $L^2(\mathcal{C}_{y} \setminus \mathcal{C}_{y^\prime})$,
$\mathcal{C}_{y} \setminus \mathcal{C}_{y^\prime}$ being precompact.

Now consider $\chi \in C^\infty(\mathcal{C})$ supported in $\mathcal{C}_{\bar{a}/4}$, with $0 \leq \chi \leq 1$, and $\chi = 1$ on $\mathcal{C}_{\bar{a}/2}$. 
It is evident that $\chi \tilde{\vf}_n \rightarrow 0$ in $L^2(\mathcal{C})$, and thus, the functions $\psi_n = (1-\chi) \tilde{\vf}_n$ satisfy $\Ray\psi_n \rightarrow 1/4$. Indeed, the support of $\nabla \chi$ is compact and $\tilde{\phi}_n \to 0$ in $L^2(K)$ for $K$ compact, and so the Rayleigh quotients do not see the contributions from $\tilde{\phi}_n \nabla \chi$ for $n$ sufficiently large.
From the fact that $\psi_n = 0$ in $\mathcal{C}_{\bar{a}}$, we readily see that $\psi_n$ is supported in the annulus $A_n$ with boundary components $\tilde{c} = \partial \mathcal{C}_{a_i}$ and $\tilde{c}_n = \partial \tilde{C}_n$.
Moreover, since $\psi_n = 0$ in $\mathcal{C}_{\bar{a}/2}$ and $\tilde{c}_n \cap \mathcal{C}_{\bar{a}} \neq \emptyset$,
we derive that the support of $\psi_n$ does not intersect a shortest geodesic from $\tilde{x}_0$ to $\tilde{c}_n$,
a curve joining the connected components of $\partial A_n$.
This implies that the support of $\psi_n$ is contained in a topological disc of area at most $|S|$.
Since these can be lifted isometrically to $\mathbb{H}$, we arrive at a contradiction, in view of the Faber-Krahn inequality.

The second asserted inequality is clear:
Since the competing domains are compact with piecewise smooth boundary and $S$ is non-compact, the complement of each of the domains is a non-empty open subset of $S$.
Hence there is room for enlarging the domains and thereby diminishing $\lambda_0$.
\end{proof}

\begin{lem}\label{cusp estimate}
Let $C_{i,b_i}$ be a cusp of $S$ and $f \in \Lip(C_{i,b_i}) \cap H^1(C_{i,b_i})$ a non-zero function that vanishes at some point of $H_{i,t}$ for any $t \geq b_i$. Then $\Ray(f) \geq \pi^2 b_i^2$.
\end{lem}

\begin{proof}
Using the co-area formula, we compute
\begin{eqnarray}
\int_{C_{i,b_i}} |\nabla f|^2 &=& \int_{b_i}^{+\infty} t^{-1} \int_{H_{i,t}} |\nabla f|^2 dt \geq  \int_{b_i}^{+\infty} t^{-1} \int_{H_{i,t}} |\nabla^\top f|^2 dt \nonumber\\
&=& \int_{b_i}^{+\infty} t^{-1} \int_{H_{i,t}} |\nabla f|_{H_{i,t}}|^2 dt, \nonumber
\end{eqnarray}
where $\nabla^\top f$ stands for the tangential component of $\nabla f$ to $H_{i,t}$. Keeping in mind that $f|_{H_{i,t}}$ vanishes at some point and $H_{i,t}$ is the circle of length $t^{-1}$, $f|_{H_{i,t}}$ can be lifted to a function on $[0,t^{-1}]$ vanishing on the boundary, with the same Rayleigh quotient. Therefore,
\[
 \int_{H_{i,t}} |\nabla f|_{H_{i,t}}|^2 \geq \lambda_0^D([0,t^{-1}]) \int_{H_{i,t}} f^2
\]
for any $t \geq b_i$, where $\lambda_0^D([0,t^{-1}]) = \pi^2 t^2$ stands for the bottom of the Dirichlet spectrum of $[0,t^{-1}]$ (with the Euclidean metric). Combining the above, we conclude that
\[
\int_{C_{i,b_i}} |\nabla f|^2 \geq \pi^2 b_i^2 \int_{b_i}^{+\infty} t^{-1} \int_{H_{i,t}} f^2 dt = \pi^2 b_i^2 \int_{C_{i,b_i}} f^2,
\]
exploiting again the co-area formula.
\end{proof}

 \section{ Proof of \cref{main}}\label{secnod}
Recall that $S$ is a complete hyperbolic surface without boundary and with $n\ge1$ cusps $C_1,\dots,C_n$,
where $C_i=\Gamma_i\backslash\sigma_i\{y\geq b_i\}$ as in \cref{seceis}.
Here we discuss the structure of nodal sets of linear combinations of eigenfunctions of $\Delta_a$,
where $a=(a_1,\dots,a_n)$ is an $n$-tuple of real numbers with $a_i>b_i$.

Let $u$ be a non-trivial (finite) linear combination of eigenfunctions of $\Delta_a$ (for various eigenvalues).
 Then we may write
 \begin{equation}
 u = \varphi + \sum\nolimits_s d_s\hat f_s,
\end{equation}
 where $\varphi$ is a linear combination of cuspidal eigenfunctions and the $\hat f_s$ are truncations of functions $f_s$ as in \eqref{eisres} and \eqref{eispar}.
 Let
 \begin{align*}
     \tilde{u} = \varphi + \sum\nolimits_s d_sf_s,    
 \end{align*}
 defined and analytic on all of $S$ and
 \begin{align*}
    \hat{u} = \tilde{u} - \sum\nolimits_s d_s[f_s]_{a-\ve},
 \end{align*}
 defined and analytic on
 \begin{align*}
     C_{a-\ve} = \cup_{i=1}^n C_{i,a_i -\ve} \subset S
 \end{align*}
 for $a_i-\ve > b_i$.
 Notice that
 \begin{align*}
     \tilde u|_{S_a} = u|_{S_a}, \quad \hat u|_{C_a} = u|_{C_a}, \quad\text{and}\quad \tilde u|_{H_a} = \hat u|_{H_a} = u|_{H_a}.
 \end{align*}
 Notice also that $\tilde u=0$ would imply $\varphi=0$ and $\sum_sd_sf_s=0$. But then we would also have $\hat u=0$, and then $u=0$, a contradiction.
 On the other hand, $\hat u=0$ would imply that $\tilde u=\sum_s d_s [f_s]_{a-\ve}$ on $C_{a-\ve}$.
But then $\tilde u$ would only depend on the $y$-parameter in the cusps and would not be analytically extendable to the rest of $S$.
Hence both, $\tilde u$ and $\hat u$, do not vanish.

\begin{prop}\label{prop:eigen-nodal-set}
Let $u$ be a non-trivial finite linear combination of eigenfunctions of $\Delta_a$.
Then $Z(u)$ is a locally finite graph. Moreover, if $u$ is a finite linear combination of eigenfunctions of $\Delta_a$ then 
$S\setminus Z(u)$ has at least two components.
 \end{prop}

Notice that we do not exclude the possibility that $Z(u)$ contains vertices of degree zero and of odd degree.
 Those of odd degree may occur along $H_a$, but not elsewhere.

 \begin{proof}[Proof of \cref{prop:eigen-nodal-set}] 
If $u=\varphi$ is cuspidal, then $\tilde u=\hat u=u$ is analytic on $S$ and the assertion follows from \cite[Proposition 4]{OR}.
 In this case, all vertices of $Z(u)$ are of even degree, where vertices of degree zero, that is, isolated points, are not excluded.
 In general, $u$ is analytic on $S\setminus H_a$ and hence, again by \cite[Proposition 4]{OR} and away from $H_a$,
 $Z(u)$ is a locally finite graph with vertices of even degree.
 By what we said above,
 \begin{align*}
     Z(u) = Z(\tilde{u}|_{S_a}) \cup Z(\hat{u}|_{C_{a}}).
 \end{align*}
 Moreover, $\tilde u$ and $\hat u$ are analytic on $S$ and $C_{a-\ve}$, respectively,
 and hence, again by \cite[Proposition 4]{OR},
 $Z(\tilde u)$ and $Z(\hat u)$ are locally finite graphs with vertices of even degree on $S$ and $C_{a-\ve}$, respectively.
 Now $u|_{H_a}$ is analytic.
 Hence, for each $1\le i\le n$, there are now two cases:
 Either $u|_{H_{i,a_i}}=0$ or else $u$ has only finitely many zeros along $H_{i,a_i}$.
 For each such zero $x\in H_{i,a_i}$, there is a finite number of edges of $Z(\tilde{u}|_{S_a})$ and $Z(\hat{u}|_{C_{a}})$ ending in $x$.
 In both cases, $Z(u)$ is a graph about $x$,
 and $x$ is a vertex of degree the sum of the degrees of $x$ as a vertex of $\tilde{u}|_{S_a}$ and of $\hat{u}|_{C_{i,a_i}}$,
 diminished by two in the first case.
 
 If $u$ is a finite linear combination of eigenfunctions of $\Delta_a$, then $u \in D(\Delta_a)$, and so the zeroth Fourier coefficient of $u$ is zero in $H_a$. Therefore,
 \[
 \int_{H_a} u  = 0.
 \]
 This implies the last conclusion.
 \end{proof}

Let $\vf$ be a finite linear combination of eigenfunctions of $\Delta_a$.
By above discussion, $\vf$ is a continuous function on all of $S$, and $\vf|_{S_a}, \vf|_{C_a}, \vf|_{H_a}$ are analytic functions.
By Sard's Theorem, almost all values of these functions are regular values.
Following \cite{BMM} we say $\ve >0$ is $\vf$ regular if $\pm \ve$ are regular values of each of these three functions.
 Let 
 \[
 Z_\vf(\ve) = \{ x \in S: |\vf(x)| \le \ve\}
 \]
denote the $\ve$-nodal set of $\vf$.

Clearly, $\partial Z_\vf(\ve) \cap S_a$ and $\partial Z_\vf(\ve) \cap C_a$ consists of smooth arcs.
By regularity of $\vf|_{H_a}$, for any $p \in H_a$, at most one arc in $S_a$ (and at most one arc in $C_a$) can end at $p$.
By continuity of $\vf$, these arcs meet along $H_a$, and forms continuous closed arcs. 
Finally, $\cap_{\ve>0} Z_\vf(\ve) = Z(\vf).$

 \begin{prop}\label{prop:eigen-approxnodal-set}
 Let $\vf$ be a finite linear combination of eigenfunctions of $\Delta_a$. 
 For any $\vf$-regular $\ve>0$, the boundary $\partial Z_\vf(\ve)$ of the $\ve$-nodal set $Z_\vf(\ve)$ consists of a locally finite disjoint collection of continuous, piece-wise analytic loops. Moreover, only at points on $H_a$ these loops are possibly not analytic, and at any such point the two arcs emanating makes a non-zero angle.
 \end{prop}
 
 \begin{proof}
 Local finiteness of $\partial Z_\vf(\ve)$ near $H_a$ is explained above. At other points, $\vf$ is analytic.
 
 For any $p \in H_a$, there is at most one arc in in $\partial Z_\vf(\ve)\cap S_a$ (and at most one arc in $\partial Z_\vf(\ve)\cap C_a$) that can emanate from $p$, and these arcs cannot be tangential to $H_a$, because $\vf|_{H_a}$ is a regular at $p$.
 \end{proof}

\subsection{The Otal-Rosas argument} 
The proof of \cref{main} involves arguments from surface topology and follows the approach in \cite{BMM} very closely.
In fact, we refer the reader to \cite[Section 4]{BMM} for most of the arguments, and comment on the changes necessary to make the arguments there work. 

To begin, we consider a finite linear combination $\vf$ of eigenfunctions of $\Delta_a$ and observe that $\vf$ satisfies \cite[Lemma 4.3]{BMM}.
Indeed, the proof works almost identically, because $\vf \in H^1(S)$ and the points where $\vf$ is possibly not differentiable is contained in a union of horocycles, and so has zero measure. 

Next, we observe that the notion of $\ve$-discs carries over almost identically for any $\vf$ as above. 
There is an ambiguity about the definition of $\nu$ along the horocycle $H_a$. 
To remedy this, we observe that $H_a \cap \partial Z_\vf(\ve)$ is discrete, and define $\nu|_{H_a} = \partial \theta$, where $\theta$ is the unit tangent vector along $H_a$. 
If an arc in $\partial Z_\vf(\ve)$ intersects $H_a$, it does so at a non-zero angle; this definition is particularly well-suited for our purpose.
In particular, all the discussions in \cite{BMM} on $\ve$-discs go through, and one has the conclusions of Lemma 4.6-4.8 from \cite{BMM} for any $\vf$ that is a finite linear combination of eigenfunctions of $\Delta_a$.

Next we consider \cite[Lemma 4.10]{BMM} and observe that the arguments in the proof work almost identically to provide the existence of a component $Y_\vf(\ve)$ whose fundamental group contains the free group $F_2$. 
Indeed, the only difference is in the case where we consider a $\lambda$-eigenfunction $\vf$ with $\lambda \le \Lambda_a(S)$. 
In this case, unlike in \cite{BMM}, we do not have smoothness of $\vf$. 
However, since $\vf$ is an eigenfunction of $\Delta_a$, by \cref{prop:eigen-nodal-set}, we know that $Z(\vf)$ is a locally finite graph, and $S \setminus Z(\vf)$ has more than one component. 
Thus, by the last part of \cref{prop:analytic-systole}, no component of $S \setminus Z(\vf)$ is a disc, an annulus, or a M\"obius band. 
The proof for the incompressibility of the components in \cite[Lemma 4.10]{BMM} works without any change to give us the rest of the assertions of \cite[Lemma 4.10]{BMM}.

The rest of the arguments in \cite{BMM} leading to the proof of the main theorem \cite[Theorem 1.5]{BMM} are topological in nature, and works without any change (other than using the above mentioned modified versions of Lemma 4.3-4.10), to give a proof of \cref{main}.


\end{document}